\documentclass[leqno,12pt]{article} %leqno is the option to put formula numbers on the left side
\usepackage{amsmath, amssymb}
\usepackage{amsthm} %theorem environment option
\theoremstyle{plain} %text of this environment is typesetted in italics
\newtheorem{theorem}{\indent\sc Theorem}[section]
\newtheorem{lemma}[theorem]{\indent\sc Lemma}
\newtheorem{corollary}[theorem]{\indent\sc Corollary}
\newtheorem{proposition}[theorem]{\indent\sc Proposition}

\theoremstyle{definition} %text of this environment is typesetted in roman letters
\newtheorem{definition}[theorem]{\indent\sc Definition}
\newtheorem{remark}[theorem]{\indent\sc Remark}
\newtheorem{example}[theorem]{\indent\sc Example}
\newtheorem{examples}[theorem]{\indent\sc Examples}
\newtheorem{notation}[theorem]{\indent\sc Notation}
\newtheorem{note}[theorem]{\indent\sc Note}

\newcommand{\R}{{{\mathcal R}}}

\newcommand{\MM}{{\mathbb M}}

\DeclareMathOperator{\Hom}{{Hom}}
\DeclareMathOperator{\Spec}{{Spec}}
\DeclareMathOperator{\Coker}{{Coker}}
\DeclareMathOperator{\Ker}{{Ker}}
\DeclareMathOperator{\Ima}{{Im}}
\DeclareMathOperator{\Der}{{Der}}

\newcommand{\leftmapsto}{\leftarrow\!\shortmid}
\newcommand{\dosflechasa}[3][]{\xymatrix@1{\ar@<1ex>[r]^-{#2}
\ar@<-1ex>[r]_-{#3} & }}
\newcommand{\dosflechas}{{\dosflechasa[]{}{}}}

\DeclareMathOperator{\limi}{{lim}}
\newcommand{\ilim}[1]{\,\underset{#1}{\underset{\to}{\limi}}\,}
\newcommand{\plim}[1]{\,\underset{#1}{\underset{\leftarrow}{\limi}}\,}

\usepackage[all]{xy}

\title{\uppercase{Geometric characterization of flat modules}} %title of the paper
\author{
\textsc{Carlos Sancho,  Fernando Sancho and  Pedro Sancho} %names of authors
}
\date{} %leave empty
\begin{document}

\maketitle

%%%%%%%%%%%%%%% footnote %%%%%%%%%%%%%%%%
\footnote{ %2010 MSC numbers
2010 \textit{Mathematics Subject Classification}.
Primary 13C11; Secondary 18Axx.
}
\footnote{ %key words and phrases
\textit{Flat modules, Mittag-Leffler, functor of modules}.
}
%\footnote{ %acknowledgment of support etc. if any
%$^{*}$Thanks.
%}
%%%%%%%%%%%%%%%%%%%%%%%%%%%%%%%%%%%%%%%%%
\begin{abstract}

%
%\newtheorem{theorem}{Theorem}[section]
%\newtheorem{lemma}[theorem]{Lemma}
%\newtheorem{proposition}[theorem]{Proposition}
%\newtheorem{corollary}[theorem]{Corollary}
%\newtheorem{remark}[theorem]{Remark}
%\newtheorem{note}[theorem]{Note}
%\newtheorem{caution}[theorem]{Caution}
%\newtheorem{notes}[theorem]{Notes}
%\newtheorem{Formula of adjoint functors}[theorem]{Formula of adjoint functors}
%
%
%\newtheorem{example}[theorem]{Example}
%\newtheorem{examples}[theorem]{Examples}
%\newtheorem{definition}[theorem]{Definition}
%\newtheorem{notation}[theorem]{Notation}
%\newtheorem{notations}[theorem]{Notations}
%\newtheorem{Adjunction formula}[theorem]{\indent\sc Adjunction formula}
%
%\DeclareMathOperator{\limi}{{lim}}
%\newcommand{\ilim}[1]{\,\underset{#1}{\underset{\to}{\limi}}\,}
%\newcommand{\plim}[1]{\,\underset{#1}{\underset{\leftarrow}{\limi}}\,}
%

%\keywords{Flat modules, flat Mittag-Leffler modules, functor of points}

Let $R$ be a commutative ring.  Roughly speaking, we prove that an $R$-module $M$ is flat iff it is a direct limit of  $R$-module affine algebraic varieties, and $M$ is a flat Mittag-Leffler module iff it is  the union of its   $R$-submodule affine algebraic varieties.

\end{abstract}

\maketitle

\section*{Introduction}

Let $R$ be a commutative (associative with unit) ring. All the functors considered in this paper are covariant functors from the category of commutative $R$-algebras  to the category of sets (groups, rings, etc.).

Let $\R$ be the functor of rings defined by  $\R(R'):=R'$ for any commutative  $R$-algebra $R'$.  An $\R$-module $\MM$ is a functor of abelian groups endowed with a morphism of functors
$$\R\times \MM\to\MM$$ satisfying the module axioms (in other words, the morphism $\R\times \MM\to\MM$ yields an $R'$-module structure on $\MM(R')$ for any commutative $R$-algebra $R'$). Morphisms of $\R$-modules, direct and inverse limits and  tensor products, etc., are defined in the obvious way.

Each $R$-module $M$ defines an $\R$-module $\mathcal M$ in the following way: $\mathcal M(R'):=M\otimes_R R'$. $\mathcal M$  is said to be a  {\sl quasi-coherent} $\R$-module. The functors
$$\aligned \text{Category of $R$-modules } & \to \text{ Category of quasi-coherent $\R$-modules }\\ M & \mapsto \mathcal M\\ \mathcal M(R) & \leftmapsto \mathcal M\endaligned$$
stablish an equivalence  of categories.
 Consider the dual functor
$\mathcal M^*:=\mathbb Hom_{\R}(\mathcal M,\R)$ defined by
$\mathcal M^*(R'):=\Hom_{R'}(M\otimes_RR',R')$.   In general, the canonical morphism  $M\to M^{**}$ is not an isomorphism, but surprisingly  $\mathcal M=\mathcal M^{**}$ (see \ref{reflex}), that is, $\mathcal M$ is a reflexive $\R$-module. 
This result has many applications in Algebraic Geometry
 (see \cite{Pedro}), for example the Cartier duality of commutative affine groups and commutative formal groups. 
 
Each $R$-scheme $X$ defines ``the functor of points''  $X^\cdot$ as follows: 
$$X^\cdot (R'):=\Hom_{\Spec R'}(\Spec R', X\times_{\Spec R}\Spec R').$$ It is well known that
$$\Hom_{sch}(X,Y)=\Hom_{funct}(X^\cdot,Y^\cdot),$$ and $X\simeq Y$ iff 
$X^\cdot\simeq Y^\cdot$. Let $\mathbb A^1_R:=\Spec R[x]$ be the affine line. An $R$-scheme $X$ is said to be an $R$-module scheme if there exist  morphisms $$X\times_{\Spec R }X\to X,\quad {\mathbb A}_{R}^1\times_{\Spec R} X\to X$$
satisfying the module axioms. It is well known that $X$ is an 
$R$-module scheme iff $X^\cdot$  is an $\R$-module (we shall say that  $X^\cdot$ is an $\R$-module scheme). 
For example, $\mathcal M^*$ is an $\R$-module (affine) scheme
because $ \mathcal M^*=(\Spec S^\cdot M)^\cdot$. We prove the following proposition.

\begin{proposition} (\ref{T3.3}) An $\mathcal R$-module scheme $X^\cdot$ is isomorphic to $\mathcal M^*$, for some $R$-module $M$, iff $X^\cdot$ is a reflexive $\R$-module affine scheme. 
\end{proposition}

An $R$-module $M$ is a  finitely generated projective $R$-module iff $\mathcal M$ is an $\R$-module quasi-compact and quasi-separated  scheme (see \ref{PFS}). 
The Govorov-Lazard Theorem states that an  $R$-module $M$ is flat iff $\mathcal M$ is isomorphic to a direct limit (of functors of points of) of affine n-spaces, $\mathbb A^n_R$  (with the natural structure of $\R$-module). If $R$ is a field and $\mathbb M$ is a reflexive $\mathcal R$-module,  then  $\mathbb M$ is isomorphic to a direct limit $\ilim{i} \mathcal N_i^*$ (see \cite{Pedro}). 
 In this paper we prove the following theorem.

\begin{theorem}   Let $M$ be an $R$-module. The following statements are equivalent
\begin{enumerate}

\item $M$ is a flat $R$-module.

\item $\mathcal  M$ is a direct limit of $\R$-module schemes.

\item $\mathcal M=\ilim{i} \mathcal N_i^*$, that is, $\mathcal M$ is a direct limit of reflexive $\R$-module affine schemes.

\end{enumerate}

\end{theorem}

 Given an $\R$-module $\mathbb M$,  consider the reflexive $\R$-module affine scheme   $$ \mathbb M_{sch}:=(\Spec S^\cdot \mathbb M^*(R))^\cdot.$$ There exist a natural morphism $\mathbb 
 M\to \mathbb M_{sch}$ and a functorial equality
 $$\Hom_{\R}(\mathbb M_{sch},\mathcal N^*)=\Hom_{\R}(\mathbb M,\mathcal N^*),$$
 for any reflexive $\R$-module affine scheme $\mathcal N^*$.
 Let $R$ be a Noetherian ring, $M$ an $R$-module and  $\{M_i\}_{i\in I}$ the set of the finitely generated  submodules of  $M$. We prove that $M$ is flat iff  $$\mathcal M=\ilim{i\in I} \mathcal M_{i,sch}.$$
 In other words (see \ref{NMP}), $M$ is flat iff
 $$M\otimes_R N=\ilim{i}\Hom_R(M_i^*,N),$$
 for any $R$-module $N$.
 In \ref{elteorema}, we generalize this result to characterize flat quasi-coherent sheaves on Noetherian schemes. Let $M_i':=\Ima[M^*\to M_i^*]$. In a next paper we will prove that $M$ is a strict flat Mittag-Leffler $R$-module iff $$M\otimes_R N=\ilim{i}\Hom_R(M_i',N),$$
 for any $R$-module $N$.

Finally, we characterize flat  Mittag-Leffler modules.
Mittag-Leffler conditions were first introduced by Grothendieck in \cite{Grot}, and deeply studied
by some authors, for example, Raynaud and Gruson in \cite{RaynaudG}. $M$ is said to be a flat   Mittag-Leffler module if it is a direct limit of free finite $R$-modules $\{L_i\}$ such that the inverse system $\{L_i^*\}$
 satisfies the usual Mittag-Leffler condition: for each $i$ there exists $j\geq  i$ such that for $k\geq j$ we have $\Ima(L_k^*\to L_i^*)=\Ima(L_j^*\to L_i^*)$, (see \cite{Raynaud}).

\begin{theorem} \label{MLI} Let $M$ be an $R$-module. The following statements are equivalent

\begin{enumerate}
\item $M$ is a flat Mittag-Leffler module..

\item $\mathcal M$ is equal to the direct limit of its $\R$-submodule affine schemes.

\item $\mathcal M=\ilim{i} \mathcal N_i^*$, where the morphism $\mathcal N_i^*\to \mathcal N_j^*$ is a monomorphism for any $i\leq j$, that is, $\mathcal M$  is equal to the direct limit of its reflexive $\R$-submodule affine schemes.

\item The kernel of any morphism $\mathcal N^*\to \mathcal M$ 
is a reflexive $\mathcal R$-module affine scheme.
 
\item The kernel of any morphism $\mathcal R^n\to \mathcal M$ 
is a reflexive $\mathcal R$-module affine scheme.

\item Any morphism $\mathcal N^*\to \mathcal M$ factors through an $\R$-submodule scheme $\mathcal W^*$ of $\mathcal M$.

\item Any morphism $\mathcal R^n\to \mathcal M$ factors through an $\R$-submodule scheme $\mathcal W^*$ of $\mathcal M$.

\end{enumerate}
\end{theorem}

%Theorem \ref{MLI} (5) can be easily rewritten as Corollary \ref{MLJ} (2), that is, we have the following corollary. 
%
%\begin{corollary} \label{MLJ} Let $M$ be an $R$-module. The following statements are equivalent
%
%\begin{enumerate}
%\item $M$ is a flat Mittag-Leffler module. 
%
%\item For any $R$-module $P$ and any $\sum_{i=1}^n m_i\otimes p_i\in M\otimes_R P$, $\sum_{i=1}^n m_i\otimes p_i=0$ iff
%$\sum_{i=1}^n w(m_i)\cdot p_i=0$, for every $w\in M^*$. \end{enumerate}\end{corollary}

There exists a natural isomorphism $N\otimes_R M\overset{\text{\ref{prop4}}}=\Hom_{\mathcal R}(\mathcal N^*,\mathcal M)$. 
Then, for example, \ref{MLI}.7. can be rewritten as follows: 

$(7')$ For any $x\in R^n\otimes_R M$ there exist a module $W$, a morphism $\phi\colon W\to  R^n$ and $x'\in W\otimes_R M$ satisfying:
\begin{enumerate}
\item[(a)] $(\phi \otimes Id)(x')=x$ (where $\phi\otimes Id\colon W\otimes M\to R^n\otimes M$ is the morphism defined by $(\phi\otimes Id)(w\otimes m):=\phi(w)\otimes m$).

\item[(b)] $(\phi'\otimes Id)(x')\neq 0$, for any non-zero morphism $\phi'\colon W\to W'$.
\end{enumerate} Taking $W=\Ima \phi$ in $(7')$,  we can say that $W$ is a submodule of $R^n$. If $M$ is assumed to be flat, then $(a)$ and $(b)$ are equivalent to $W$ being the smallest submodule of $R^n$ that satisfies $(a)$. Therefore, a flat module $M$ is a flat Mittag-Leffler module iff for any
$x\in R^n\otimes_R M$  there exists the smallest submodule $W\subset R^n$ such that $x\in W\otimes_R M$ (this statement was proved in \cite{RaynaudG}). 

No knowledge on flat Mittag-Leffler modules is required in this paper. We show how several known  properties of flat   Mittag-Leffler modules can be proved as consequences of \ref{MLI}.3.  (see \ref{CML}, \ref{CMLC} and \ref{CMLD})). In a next paper (see \cite{Pedro3}), we shall give some functorial characterizations of flat strict Mittag-Leffler modules (see \cite{Garfinkel} for a definition).

Quasi-coherent modules, $\mathcal M$, and reflexive $\mathcal R$-module affine schemes, $\mathcal N^*$, are the two most widely used  functors  in this paper. We think that the main results of this paper involving only these functors are equally true when $R$ is not commutative.

\section{Preliminaries}

\begin{notation} \label{nota2.1}For simplicity, given a (covariant) functor  $\mathbb X$ (from the category of commutative $R$-algebras to the category of sets),
we shall sometimes use $x \in \mathbb X$  to denote $x \in \mathbb X(R')$. Given $x \in \mathbb X(R')$ and a morphism of commutative $R$-algebras $R' \to R''$, we shall still denote by $x$ its image by the morphism $\mathbb X(R') \to \mathbb X(R'')$.\end{notation}

 Let $\mathbb M$ and $\mathbb M'$ be two $\R$-modules.
 A morphism of $\R$-modules $f\colon \mathbb M\to \mathbb M'$
 is a morphism of functors  such that the  morphism $f_{R'}\colon \mathbb M({R'})\to
 \mathbb M'({R'})$ defined by $f$ is a morphism of ${R'}$-modules, for any commutative $R$-algebra ${R'}$.
 We shall denote by $\Hom_{\R}(\mathbb M,\mathbb M')$ the  family of all the morphisms of $\R$-modules from $\mathbb M$ to $\mathbb M'$.

\begin{remark} Direct limits, inverse limits of $\R$-modules and  kernels, cokernels, images, etc.,  of morphisms of $\R$-modules are regarded in the category of $\R$-modules.\end{remark}

One has
$$\aligned & 
(\Ker f)({R'})=\Ker f_{R'},\, (\Coker f)({R'})=\Coker f_{R'},\, (\Ima f)({R'})=\Ima f_{R'},\\
&  (\ilim{i\in I} \mathbb M_i)({R'})=\ilim{i\in I} (\mathbb M_i({R'})),\,
(\plim{j\in J} \mathbb M_j)({R'})=\plim{j\in J} (\mathbb M_j({R'})),\endaligned $$
(where $I$ is an upward directed set and $J$ a downward directed set).
$\mathbb M\otimes_{\R}\mathbb M'$ is defined by $(\mathbb M\otimes_{\R}\mathbb M')({R'}):=\mathbb M({R'})\otimes_{{R'}}\mathbb M'({R'})$, for any commutative $R$-algebra ${R'}$.

\begin{definition}  Given an $\R$-module $\mathbb M$ and a commutative $R$-algebra ${R'}$,   we shall denote by $\mathbb M_{|{R'}}$  the restriction of $\mathbb M$ to the category of commutative ${R'}$-algebras, i.e.,  $$\mathbb M_{\mid {R'}}(R''):=\mathbb M(R''),$$ for any commutative ${R'}$-algebra $R''$.
\end{definition}

We shall denote by ${\mathbb Hom}_{\R}(\mathbb M,\mathbb M')$\footnote{In this paper, we shall only  consider well-defined functors ${\mathbb Hom}_{\R}(\mathbb M,\mathbb M')$, that is to say, functors such that $\Hom_{\mathcal {R'}}(\mathbb M_{|{R'}},\mathbb {M'}_{|{R'}})$ is a set, for any ${R'}$.} the  $\R$-module defined by $${\mathbb Hom}_{\R}(\mathbb M,\mathbb M')({R'}):={\rm Hom}_{\mathcal {R'}}(\mathbb M_{|{R'}}, \mathbb M'_{|{R'}}).$$  Obviously,
$$(\mathbb Hom_{\R}(\mathbb M,\mathbb M'))_{|{R'}}=
\mathbb Hom_{\mathcal {R'}}(\mathbb M_{|{R'}},\mathbb M'_{|{R'}}).$$

\begin{notation} \label{nota2.2} Let $\mathbb M$ be an $\R$-module. We shall denote $\mathbb M^*=\mathbb Hom_{\R}(\mathbb M,\R)$.\end{notation}

\begin{proposition} \label{trivial} Let $\mathbb M$ and $\mathbb N$ be two  $\R$-modules. Then,
$$\Hom_{\R}(\mathbb M,\mathbb N^*)=\Hom_{\R}(\mathbb N,\mathbb M^*),\, f\mapsto \tilde f,$$
where $\tilde f$ is defined as follows: $\tilde f(n)(m):=f(m)(n)$, for any $m\in\mathbb M$ and $n\in\mathbb N$.

\end{proposition}

\begin{proof} $\Hom_{\R}(\mathbb M,\mathbb N^*)=\Hom_{\R}(\mathbb M\otimes_{\R}\mathbb N,\R)=\Hom_{\R}(\mathbb N,\mathbb M^*)$.

\end{proof}

\begin{proposition} \cite[1.15]{Amel}  \label{adj2} Let $\mathbb M$ be an $\R$-module, ${R'}$ a commutative $R$-algebra and $N$ an ${R'}$-module.  Then,
$$\Hom_{\mathcal {R'}}(\mathbb M_{|{R'}},\mathcal N)=\Hom_{\R}(\mathbb M,\mathcal N).$$
In particular,
$$\mathbb M^*({R'})=\Hom_{\R}(\mathbb M,\mathcal {R'}).$$
\end{proposition}

\subsection{Quasi-coherent modules}

\begin{definition} Let $M$ (resp. $N$,  $V$, etc.) be an $R$-module. We shall denote by  ${\mathcal M}$  (resp. $\mathcal N$, $\mathcal V$, etc.)   the $\R$-module defined by ${\mathcal M}({R'}) := M \otimes_R {R'}$ (resp. $\mathcal N({R'}):=N\otimes_R {R'}$,  $\mathcal V({R'}):=V\otimes_R {R'}$, etc.). $\mathcal M$  will be called the quasi-coherent $\R$-module associated with $M$. 
\end{definition}

 ${\mathcal M}_{\mid {R'}}$ is the quasi-coherent $\mathcal {R'}$-module associated with $M \otimes_R
{R'}$.  For any pair of $R$-modules $M$ and $N$, the quasi-coherent module associated with $M\otimes_R N$ is $\mathcal M\otimes_{\R}\mathcal N$.

\begin{proposition} \cite[1.12]{Amel}  The functors
$$\aligned \text{Category of $R$-modules } & \to \text{ Category of quasi-coherent $\R$-modules }\\ M & \mapsto \mathcal M\\ \mathcal M(R) & \leftmapsto \mathcal M\endaligned$$
stablish an equivalence  of categories. In particular,
$${\rm Hom}_{\R} ({\mathcal M},{\mathcal M'}) = {\rm Hom}_R (M,M').$$
\end{proposition}

Let $f_R\colon M\to N$ be a morphism of $R$-modules and $f\colon \mathcal M\to \mathcal N$
the associated morphism of $\mathcal R$-modules. Let $C=\Coker f_R$, then $\Coker f=\mathcal C$, which is a quasi-coherent module.

\begin{proposition} \cite[1.3]{Amel}\label{tercer}
For every  ${\R}$-module $\mathbb M$ and every $R$-module $M$, it is satisfied that
$${\rm Hom}_{\R} ({\mathcal M}, \mathbb M) = {\rm Hom}_R (M, \mathbb M(R)),\, f\mapsto f_R.$$
\end{proposition}

\begin{notation} Let $\mathbb M$ be an $\R$-module. 
We shall denote by $\mathbb M_{qc}$ the quasi-coherent module associated with
the $R$-module $\mathbb M(R)$, that is, $$\mathbb M_{qc}({R'}):=\mathbb M(R)\otimes_R{R'}.$$\end{notation}

\begin{proposition} \label{tercerb} For each  $\R$-module $\mathbb M$ one has the natural morphism $$\mathbb M_{qc}\to \mathbb M, \,m\otimes r'\mapsto r'\cdot m,$$ for any $m\otimes r'\in \mathbb M_{qc}({R'})=\mathbb M(R)\otimes_R {R'}$, and a functorial equality 
$$\Hom_{\R}(\mathcal N,\mathbb M_{qc}))=\Hom_{\R}(\mathcal N,\mathbb M),$$
for any quasi-coherent $\R$-module $\mathcal N$.
\end{proposition}

\begin{proof} Observe that $\Hom_{\R}(\mathcal N,\mathbb M)\overset{\text{\ref{tercer}}}=\Hom_R(N,\mathbb M(R))\overset{\text{\ref{tercer}}}=\Hom_{\R}(\mathcal N,\mathbb M_{qc})$.

\end{proof}

Obviously, an $\mathcal R$-module $\mathbb M$ is a quasi-coherent module iff
the natural morphism $\mathbb M_{qc}\to \mathbb M$ is an isomorphism.

\begin{theorem}  \cite[1.8]{Amel}\label{prop4} 
Let $M$ and $M'$ be $R$-modules. Then, $${\mathcal M} \otimes_{\R} {\mathcal M'}={\mathbb Hom}_{\R} ({\mathcal M^*}, {\mathcal M'}),\, m\otimes m'\mapsto \tilde{m\otimes m'},$$
where $ \tilde{m\otimes m'}(w):=w(m)\cdot m'$, for any $w\in \mathcal M^*$.
\end{theorem}

%\begin{proof} ${\Hom}_{\R} ({\mathcal M^*}, {\mathcal M'}) \overset{\text{\ref{L5.111}}}= \overline{\mathcal M'}(M)={M} \otimes_{R} {M'}.$
%
%
%\end{proof}

If we make $\mathcal M'=\R$ in the previous theorem, we obtain the following theorem.

\begin{theorem} \cite[II,\textsection 1,2.5]{gabriel}  \cite[1.10]{Amel}\label{reflex}
Let $M$ be an $R$-module. Then, $${\mathcal M}={\mathcal M^{**}}.$$
\end{theorem}

\begin{definition} Let $\mathbb M$ be  an $\R$-module. We shall say that
$\mathbb M^*$ is a dual functor.
We shall say that  an $\R$-module  ${\mathbb M}$ is reflexive if ${\mathbb M}={\mathbb M}^{**}$.\end{definition}

\begin{example}  Quasi-coherent modules are reflexive.\end{example}

\subsection{$\mathcal R$-module schemes} For each $R$-scheme $X$ we shall denote by $X^\cdot$ the functor of sets defined by 
$$X^\cdot({R'}):=\Hom_{R-sch}(\Spec {R'},X),$$
and it will be called functor of points of $X$.

\begin{example} Let us denote $\mathbb A^1_R=\Spec R[x]$ the affine line over $\Spec R$. Then,
$$(\mathbb A^1_R)^\cdot ({R'})={R'}.$$
In other words, $(\mathbb A^1_R)^\cdot=\mathcal R$.\end{example}

\begin{definition} A module $R$-scheme is a commutative group $R$-scheme $X$ endowed with a morphism of $R$-schemes
$$\mathbb A^1_R\times_R X\to X$$
satisfying the module axioms. Morphisms of module $R$-schemes are defined in the obviuous way (homorphisms of group $R$-schemes which are compatible with the action of $\mathbb A^1_R)$. Thus we obtain the category of module $R$-schemes.

\end{definition}

\begin{remark} An $R$-scheme $X$ is a $R$-module scheme iff $X^\cdot$ is an $\mathcal R$-module. Thus we have a functor
$$\aligned \text{Module $R$-schemes} & \to \text{ $\mathcal R$-modules}\\
X & \mapsto X^\cdot\endaligned$$
which is fully faithful.

\end{remark}

\begin{definition} We say that an $\mathcal R$-module $\mathbb M$ is an $\mathcal R$-module scheme if $\mathbb M\simeq X^\cdot$ for some $R$-module scheme $X$.\end{definition}

\begin{remark} Let $M$ be an $R$-module and $S^\cdot_RM$ its symmetric algebra. 
Then $\Spec S^\cdot_RM$ is an $R$-module scheme, since
$$(\Spec S^\cdot_RM)^\cdot =\mathcal M^*.$$
\end{remark}

\begin{definition} Let $M$ be an $R$-module. $\mathcal M^*$  will be called the $\R$-module scheme associated with $M$.
\end{definition}

\begin{definition} Let $\mathbb N$ be an $\R$-module. 
We shall denote by $\mathbb N_{sch}$ the $\R$-module scheme defined by $$\mathbb N_{sch}:=((\mathbb N^*)_{qc})^*.$$\end{definition}

\begin{proposition} \label{1211b} Let $\mathbb N$ be an $\R$-module. Then,
\begin{enumerate}
\item  $\mathbb N_{sch}({R'})=\Hom_R(\mathbb N^*(R),{R'})$.

\item $
\Hom_{\R}(\mathbb N_{sch},\mathcal M)=\mathbb N^*(R)\otimes_R M$,
for any quasi-coherent module $\mathcal M$.

\end{enumerate}

\end{proposition}

\begin{proof} 1. $\mathbb N_{sch}({R'})=\Hom_{\R}((\mathbb N^*)_{qc},\mathcal {R'})=\Hom_R(\mathbb N^*(R),{R'})$.

2. $\Hom_{\R}(\mathbb N_{sch},\mathcal M)\overset{\text{\ref{prop4} }}=(\mathbb N^*)_{qc}(R)\otimes_RM=\mathbb N^*(R)\otimes_R M$.

\end{proof}

The natural morphism $(\mathbb N^*)_{qc} \to \mathbb N^*$ corresponds by Proposition \ref{trivial} with a morphism $$\mathbb N\to \mathbb N_{sch}.$$ Specifically, one has the natural morphism
$$\begin{array}{lll} \mathbb N({R'}) & \to & \Hom_{R}(\mathbb N^*(R),{R'})=\mathbb N_{sch}({R'})\\ n & \mapsto & \tilde n,\text{ where } \tilde n(w):=w_{R'}(n)\end{array}$$

\begin{proposition} \label{1211} Let $\mathbb N$ be  an $\R$-module and $M$ an $R$-module. Then, the natural morphism
 $$\Hom_{\R}(\mathbb N,\mathcal M^*)\to
\Hom_{\R}(\mathbb N_{sch},\mathcal M^*),$$
is an isomorphism.

\end{proposition}

\begin{proof} $\Hom _{\R}(\mathbb N, \!\mathcal M^*) \! \overset{\text{\ref{trivial}}} =\!
\Hom_{\R}(\mathcal M,\!\mathbb N^*)\! \overset{\text{\ref{tercerb}}}=\!\Hom_{\R}(\mathcal M,\! (\mathbb N^*)_{qc})\!{\overset{\text{\ref{trivial}}}=
\Hom_{\R}(\mathbb N_{sch},\mathcal M^*)}.$

\end{proof}

%\begin{proposition}  \label{2.16} $\overline{\ilim{i}\mathbb F_i}=\ilim{i}\bar{\mathbb F}_i$, and 
%if $\mathbb F_i$ satisfies the properties $(1,2,3)$, for every  $i$,  then $\ilim{i}\mathbb F_i$    satisfies the properties $(1,2,3)$.
%
%\end{proposition}
%
%
%
%\begin{proposition} \label{L5.11}  Let $\{\mathbb F_i\}$ be a direct system of  $\R$-modules which satisfy the properties (1,2,3). Then,
%$$\Hom_{\R}(\mathcal N^*,\ilim{i} \mathbb F_i)=
%\ilim{i}\Hom_{\R}(\mathcal N^*, \mathbb F_i).$$
%\end{proposition}
%
%\begin{proof}  By Theorem \ref{L5.111},
%$$\Hom_{\R}(\mathcal N^*,\ilim{i} \mathbb F_i)=\overline{\ilim{i} \mathbb F_i}(N)=
%\ilim{i} \bar{\mathbb F}_i(N)=
%\ilim{i}\Hom_{\R}(\mathcal N^*, \mathbb F_i).$$
%\end{proof}

\subsection{From the category of algebras to the category of modules}

\begin{definition} \label{notations}
Given an $R$-module $N$,  let $R\oplus N$ be the  commutative $R$-algebra defined by $$(r,n)\cdot (r',n'):=(rr',rn'+r'n),\,\, \forall\, r,r'\in R,\text{ and } \forall\, n,n'\in N.$$  Consider the morphism of $R$-algebras $\pi_1\colon R\oplus N\to R$,  $\pi_1(r,n):=r$, which has the obvious section of $R$-algebras $R\to R\oplus N$, $r\mapsto (r,0)$.

Given  an $\R$-module  $\mathbb F$, let $\bar{\mathbb F}$ be the functor from the category of $R$-modules to the  category of $R$-modules  defined by $\bar{\mathbb F}(N):=\Ker
[\mathbb F(\pi_1)\colon \mathbb  F(R\oplus N)\to \mathbb  F(R)]$.  
Observe that given a morphism of $R$-modules $w\colon N\to N'$, we have the morphism of $R$-algebras $\tilde w\colon R\oplus N\to R\oplus N'$, $(r,n)\mapsto (r,w(n))$, and then the morphism $\bar{\mathbb F}(w)\colon \bar{\mathbb F}(N)\to \bar{\mathbb F}(N')$, $\bar{\mathbb F}(w)(f):=\mathbb F(\tilde w)(f)$, for  any $f\in\bar{\mathbb F}(N)\subset \mathbb F(R\oplus N)$.\end{definition}

We have a canonical isomorphism $\mathbb F(R\oplus N)=\mathbb F(R)\oplus \bar{\mathbb F}(N)$. Given a morphism of functor of $\R$-modules
$\phi\colon \mathbb F_1\to \mathbb F_2$ we have the morphism $$\bar \phi\colon \bar{\mathbb F}_1\to \bar{\mathbb F}_2, \,\bar\phi_N(f):=\phi_{R\oplus N}(f),$$ for any $R$-module $N$ and $f\in \bar{\mathbb F}_1(N)\subset {\mathbb F}_1(R\oplus N)$.

\begin{examples} \label{1.30}  $\overline{\mathcal M^*}(N)=\Hom_R(M,N)$ and $\bar{\mathcal M}(N)=M\otimes_RN$.

\end{examples}

\begin{proposition}  \label{2.16} $\overline{\ilim{i}\mathbb F_i}=\ilim{i}\bar{\mathbb F}_i$.
\end{proposition}

\section{Reflexive module schemes}

Recall Definition \ref{notations}.

\begin{proposition} \label{nuevo} Let $X=\Spec A$ be an $R$-module affine scheme. Then,
$$\overline{X^\cdot}(N)=\Der_R(A,N),$$
for any $R$-module $N$.
\end{proposition}

\begin{proof} Consider the additive identity element $e\in X^\cdot(R)=\Hom_{R-alg}(A,R)$. Then any $R$-module, $N$, is an $A$-module through the morphism $e$. Let $\pi_1\colon R\oplus N\to R$ be defined by $\pi_1(r,n):=r$. Then,
$$\overline{X^\cdot}(N)=\{f\in X^{\cdot}(R\oplus N)=\Hom_{R-alg}(A,R\oplus N)\colon 
\pi_1 \circ f= e\}=\Der_R(A,N).$$

\end{proof}

\begin{proposition} \label{L5.112} Let $X$ be an $R$-module scheme.
Let $N\hookrightarrow N'$ be an injective morphism of $R$-modules. Then,  the morphism $\overline{X^\cdot}(N)\to \overline{X^\cdot}(N')$ is injective.

Let $X$ be an $R$-module affine scheme and
let $\{N_i\}_{i\in I}$ be a set of $R$-modules. The natural morphism
$$\overline{X^\cdot}(\prod_{i\in I} N_i)\to \prod_{i\in I} \overline{X^\cdot}(N_i)$$ is an isomorphism.

\end{proposition}

\begin{proof} Let $f,f'\colon \Spec(R\oplus N)\to X$ be two morphisms of $R$-schemes and assume the composite morphisms of $R$-schemes
$$\xymatrix{ \Spec(R\oplus N') \ar[r]  & \Spec(R\oplus N) \ar@<1ex>[r]^-f \ar@<-1ex>[r]_-{f'} & X}$$
are equal. Then, $f(x)=f'(x)$, for any $x\in \Spec(R\oplus N')=\Spec(R\oplus N)=\Spec R$, and the induced
composite morphisms between the rings of functions, on stalks at $x$,
$$\xymatrix{O_{X,f(x)} \ar@<1ex>[r] \ar@<-1ex>[r]  & (R\oplus N)_x \ar@{^{(}->}[r] & (R\oplus N')_x}$$
are equal (given an $R$-module $M$, we denote $M_x=(R\backslash \mathfrak p_x)^{-1}\cdot M$, where $\mathfrak p_x\subset R$ is the prime ideal associated with
$x$). Then, $f=f'$ and $\overline{X^\cdot}(N)\to \overline{X^\cdot}(N')$ is injective.

Let  $X=\Spec A$ be an $R$-module scheme. Then,

$$\overline{X^\cdot}(\prod_i N_i)\overset{\text{\ref{nuevo}}}=\Der_R(A,\prod_iN_i)=\prod_i\Der_R(A,N_i)\overset{\text{\ref{nuevo}}}=\prod_i
\overline{X^\cdot}(N_i).$$

\end{proof}

\begin{proposition} \label{3.1} Let $M$ be an $R$-module. The following statements are equivalent: \begin{enumerate}

\item $M$ is a finitely generated projective $R$-module.

\item $\mathcal M =\mathcal M_{sch}$.

\item $\mathcal M^*={\mathcal M^*}_{qc}$.

%\item $M\otimes_R N =\Hom_R(M^*,N),$
%for any $R$-module $N$.

\end{enumerate}

\end{proposition} 

\begin{proof} $(1)\Leftrightarrow (2)$ By \cite{Amel2}, $M$ is a finitely generated projective $R$-module iff $\mathcal M=\mathcal N^*$ for some $R$-module $N$.

$(2) \Leftrightarrow (3)$ $(2)$ is the dual statement of $(3)$ and vice versa. 

%$(1)   \Leftrightarrow (4)$ It is well known  (see \cite[A II p.78]{Bourbaki}). 
%(see \cite[VI 5.2]{cartan}). 
%
%$(2)  \Leftrightarrow (4)$ $\mathcal M =\mathcal M_{sch}$ iff $\overline{\mathcal M} =\overline{\mathcal M_{sch}}$, that is,
%$$M\otimes_R N=\overline{\mathcal M}(N)=\overline{\mathcal M_{sch}}(N)\overset{\text{\ref{1211b}}}=\Hom_R(M^*,N),$$
%for any $R$-module $N$.

\end{proof}

\begin{corollary} \label{PFS} Let $M$ be an $R$-module. Then, $\mathcal M$ is an $\mathcal R$-module quasi-compact and quasi-separated  scheme iff $M$ is a finitely generated projective $R$-module.
\end{corollary}

\begin{proof} $\Leftarrow)$ $\mathcal M=\mathcal M_{sch}$, by \ref{3.1}. 

$\Rightarrow)$ Let $N\to N'$ be an injective morphism of $R$-modules.
The morphism
$$M\otimes_R N=\bar{\mathcal M}(N)\to \bar{\mathcal M}(N')=M\otimes_RN'$$
is injective, by Proposition \ref{L5.112}. Therefore $M$ is flat. Let $\{M_i\}$ be the set of the finitely generated submodules of $M$. The identity morphism $\mathcal M\to \mathcal M=\ilim{i}\mathcal M_i$ factors through some $\mathcal M_i$, by Proposition \ref{L5.11}. Hence, $M=M_i$ and $M$ is a finitely generated $R$-module. Let $\pi\colon L:=R^n\to M$ be an epimorphism and $N:=\Ker\pi$. The sequence of morphisms
$$0\to \mathcal N\to \mathcal L\to \mathcal M\to 0$$
is exact because $M$ is flat. $\mathcal N$ is an $\mathcal R$-module quasi-compact and quasi-separated  scheme (it is the functor of points of a closed subset of $\Spec S^\cdot L^*$), then $N$ is a finitely generated $R$-module and
$M$ is a finitely presented flat $R$-module. By \cite[6.6]{eisenbud}, $M$ is  a finitely generated projective $R$-module.

\end{proof}

\begin{theorem}  \label{121} Let $X$ be an $R$-module affine scheme\footnote{Or let $X$ be a quasi-compact and  quasi-separated scheme and let $R$ be a field.}. 
Let $\mathbb M:=X^\cdot$.  Then, $\mathbb M^*$ is a quasi-coherent $\R$-module and $\mathbb M_{sch}=\mathbb M^{**}$.
\end{theorem}

\begin{proof} Let $\mathbb Hom(X^\cdot, \R)$ be the functor (of functions of $X$) defined by
$$\mathbb Hom(X^\cdot, \R)({R'})=\Hom(X^\cdot_{|{R'}},\mathcal {R'})=\Hom_{{R'}-sch}(X\times_R{R'},\mathbb A^1_{R'}),$$
If $X=\Spec A$, then $\mathbb Hom(X^\cdot, \R)=\mathcal A$
is a quasicoherent $\R$-module and the functor of functions of $X\times_R X$ is equal to $\mathcal A\otimes_{\R}\mathcal A$.

1. $G_m:=\Spec R[x,1/x]$ (in fact $\mathbb A^1_R$) acts  on $X$. Then, $G_m^\cdot $ (in fact, $\R$) acts on $\mathcal A$ and $A=\oplus_{n\in \mathbb N} A_n$, where
$\lambda*a_n=\lambda^n\cdot a_n$, for any $\lambda\in G_m^\cdot$ and
$a_n\in A_n$. 

2. Let $\Delta\colon A\to A\otimes_R A$ be the coproduct morphism (which is obtained from the addition morphism $X\times X\to X$). Then, $\Delta(A_n)\subseteq \oplus_{i=0}^{i=n} A_i\otimes A_{n-i}$, because $G_m$ acts linearly on $X$.

3. $A_0=R$:  Given $a_0\in A_0$, $a_0(m)=a_0(\alpha\cdot m)$, for any $m\in \mathbb M$ and $\alpha\in \R$. Then, 
$a_0(m)=a_0(0)$, for any $m\in \mathbb M$. 

Let $i>0$ and $w\in A_i$, then $w(0)=0$: $w(0)=w(\alpha*0)=\alpha^i\cdot w(0)$, for any $\alpha\in \R$, then $w(0)=0\cdot w(0)=0$.  

4. $\Delta(w)=w\otimes 1+1\otimes w$, for any $w\in A_1\colon$
$\Delta(w)=w'\otimes 1+1\otimes w'$ for some $w'\in A_1$, because
$\mathbb M$ is commutative (and $A_0=R$).  Then,
$$w(m)=w(m+0)=w'(m)+w'(0)=w'(m),$$
for any $m\in\mathbb M$, that is to say, $w=w'$. 

5. $A_1=\mathbb M^*(R)\colon$ $\mathbb M^*\subseteq \mathcal A$ and obviously $\mathbb M^*(R)\subseteq A_1$. By 4., $A_1\subseteq \mathbb M^*(R)$.

6. By change of rings, $R\to {R'}$, we have that $\mathbb M^*({R'})=(\mathbb M_{|{R'}})^*({R'})=A_1\otimes_R{R'}$. Hence,
$\mathcal A_1=\mathbb M^*$. Finally, $\mathbb M_{sch}=((\mathbb M^{*})_{qc})^*=\mathbb M^{**}$.
\end{proof}

\begin{example} Assume $R=\mathbb Z/2\mathbb Z$. Let $X$ be the obvious $R$-module scheme $\Spec R[x]$, and 
$Y=\Spec R[x]/(x^2)$ the obvious $R$-module subscheme of $X$. Let $\mathbb M=Y^\cdot$. Then,  $\mathbb M^*=\R$ and $\mathbb M^{**}=X^\cdot$.
\end{example}

Assume $R$ is a ring of characteristic $2$. Let $\mathbb M$ be the obvious functor of Abelian groups $\R$. Consider $\mathbb M$ as  an $\R$-module  as follows:
$$\lambda * m:=\lambda^2\cdot m,\, \forall \lambda\in\R \text{ and } m\in\mathbb M.$$
$\mathbb M$ is an $\R$-module scheme, but $\mathbb M$ is not the $\R$-module scheme associated with an $R$-module because $\mathbb M^*(R)=0$.

\begin{corollary} \label{T3.3} Let $X$ be an $R$-module quasi-compact and quasi-separated scheme. Assume $X$ is affine or $R$ is a field.
Then, $X^\cdot$ is the $\R$-module scheme associated with an $R$-module iff $X^\cdot$ is a reflexive $\R$-module. 
%A quasi-coherent module $\mathcal M$ is an $\R$-module affine scheme  iff $M$ is a projective module of finite type.

\end{corollary}

\begin{proof} $\Leftarrow)$ Let $\mathbb M=X^\cdot$ be reflexive. Then,   $\mathbb M=\mathbb M^{**}$ is the module scheme associated with $\mathbb M^*(R)$, by Theorem \ref{121}.

%Now, if  $\mathcal M$ is an $\R$-module affine scheme, then $\mathcal M$ is the $\R$-module scheme associated with $M^*$.  Hence,
%$M$ is a projective module of finite type by \cite{Amel2}.
%
%If $M$ is a projective module of finite type, then $\mathcal M$ is the $\R$-module scheme associated with a $R$-module by \cite{Amel2}. In particular, $\mathcal M$ is 
%an $\R$-module affine scheme.
\end{proof}

\begin{proposition} \label{3.6z} Let $\mathbb M$ be
an $\R$-module affine scheme, and $\mathbb N$ a dual $\R$-module. Then,
$$\Hom_{\R}(\mathbb M,\mathbb N)=
\Hom_{\R}(\mathbb M_{sch},\mathbb N).$$

Let $\mathbb N'$ be  an $\R$-module. A morphism $f\colon \mathbb N'\to \mathbb N$ factors through 
a morphism from $\mathbb N'$ to an $\R$-module affine scheme  iff  $f$ factors through 
the morphism $\mathbb N'\to \mathbb N'\,_{sch}$.
\end{proposition}

\begin{proof} Let us write $\mathbb N=\mathbb P^*$. Then,
$$\Hom_{\R}(\mathbb M,\mathbb N)  \overset{\text{\ref{trivial}}}=
\Hom_{\R}(\mathbb P,\mathbb M^*)\overset{\text{\ref{121}}}=\Hom_{\R}(\mathbb P,(\mathbb M^*)_{qc})\overset{\text{\ref{trivial}}} =
\Hom_{\R}(\mathbb M_{sch},\mathbb N).$$

Finally, if $f$ factors through a morphism $\mathbb N'\to \mathbb M$, where $\mathbb M$ is an $\R$-module affine scheme, then we have a commutative diagram
$$\xymatrix{\mathbb  N' \ar[r]  \ar[d] & \mathbb M \ar[r] \ar[d]  & \mathbb N\\ \mathbb N'\,_{sch} \ar[r]  & \mathbb M_{sch} \ar[ur] &  }$$
and $f$ factors through  $\mathbb N'\to \mathbb N'\,_{sch}$.
\end{proof}

\section{Characterization of flat modules}

\begin{lemma} \label{a} Let $N$ and $M$ be $R$-modules.
Given a morphism $f\colon \mathcal N^*\to \mathcal M$ of  $\R$-modules, there exist a finitely generated free module   $L$ and morphisms $g\colon \mathcal N^*\to \mathcal  L$ and $h\colon \mathcal L\to \mathcal M$, such that the diagram
$$\xymatrix{\mathcal N^* \ar[r]^-f \ar[rd]_-g & \mathcal M\\ & \mathcal L\ar[u]_-h} $$
is commutative.\end{lemma}

\begin{proof} We have $f=\sum_{i=1}^r n_i \otimes m_i\in N\otimes M\overset{\text{\ref{prop4}}}=\Hom_{\R}(\mathcal N^*,\mathcal M)$. Let $L:=R^r$, and let $\{e_i\}$ be the standard basis of $L$. Let $h\colon \mathcal L\to \mathcal M$ be defined by
  $h(e_i):=m_i$. Then,  $g:=\sum_{i=1}^r n_i\otimes e_i\in N\otimes L\overset{\text{\ref{prop4}}}=\Hom_{\R}(\mathcal N^*,\mathcal L)$ verifies $h\circ g=f$.
\end{proof}

\begin{note} \label{Nota} In particular, every morphism $f\colon \mathcal N^*\to \mathcal M$ factors through the quasi-coherent $\R$-module associated with a finitely generated  submodule   of $M$ (for example, $\Ima h_R$).

\end{note}

\begin{lemma}  \label{b} Let $N$ be a finitely presented $R$-module and let $M$ be a flat $R$-module. Every morphism $\mathcal N\to\mathcal M$ uniquely factors through $\mathcal N\to \mathcal N_{sch}$, that is,
$$\Hom_{\R}(\mathcal N_{sch},\mathcal M)\overset{\text{\ref{1211b}}}=N^*\otimes M\overset{\text{\cite[7.11]{Matsumura}}}=\Hom_{\R}(\mathcal N,\mathcal M).$$
\end{lemma}

\begin{proof} Let us recall that $N^*\otimes M=\Hom_{R}(N,M)$: 
If $N=R^n$ then it is obvious. $
\Hom_{R}(-,M)$ and $(-)^*\otimes_R M$ are contravariant left exact functors. Finally, $N$ is equal to the cokernel of a morphism between  finitely generated free modules. 
\end{proof}

\begin{proposition} \label{e6.6} \cite[6.6]{eisenbud} Let $M$ be a flat $R$-module, $N$ a finitely presented $R$-module and $f\colon \mathcal N\to \mathcal M$ a morphism of $\R$-modules. Then, there exist a finitely generated free module  $L$ and  morphisms $g\colon \mathcal N\to \mathcal L$, $h\colon\mathcal L\to \mathcal M$ such that the diagram
$$
\xymatrix{\mathcal N \ar[r]^-f \ar[rd]_-g & \mathcal M\\ & \mathcal L\ar[u]_-h}
$$
is commutative.
\end{proposition}

\begin{proof} By Lemma \ref{b} and Lemma \ref{a}, we have a commutative diagram
$$
\xymatrix{\mathcal N \ar[r]^-f \ar[rd]  &  \mathcal M &\\ & \mathcal N_{sch} \ar@{-->}[u] \ar@{..>}[r] & \mathcal L\ar@{..>}[ul]}
$$
\end{proof}

Any module is a direct limit of finitely presented modules. 
The Govorov-Lazard states that any flat module
 is a direct limit of finitely generated  free modules (see \cite[A6.6]{eisenbud}). The proof of this theorem is based on Proposition \ref{e6.6}.

\begin{theorem}  \label{T3.7}  \label{4.8z}  Let $M$ be an $R$-module. The following statements are equivalent
\begin{enumerate}

\item $M$ is a flat $R$-module.

\item $\mathcal M$ is a direct limit of module schemes associated with $R$-modules, $\mathcal  M=\ilim{i} \mathcal N_i^*$

\item $\mathcal M$ is a direct limit of $\R$-module  schemes.

\end{enumerate}

\end{theorem}

\begin{proof} $(1) \Rightarrow (2)$  $M$ is a direct limit of  a direct system of finitely presented modules, $\{N_i,f_{ij}\}$. Let $f_{ij}^{sch}\colon \mathcal N_{i,sch}\to \mathcal N_{j,sch}$ be the morphism defined by $f_{ij}$.
Let us prove that $\mathcal M=\ilim{i} \mathcal N_{i,sch}$: Each morphism $\mathcal N_i\to\mathcal M$ factors through a unique morphism $\mathcal N_{i,sch}\to\mathcal M$ by \ref{b}. The morphism $\mathcal N_{i,sch}\to\mathcal M=\ilim{i}\mathcal N_i$, factors through  a morphism $\mathcal N_{i,sch}\to \mathcal N_j$ for some $j$, by \ref{L5.11}\footnote{Or because $\Hom_{\mathcal R}(\mathcal N^*,\ilim{i}\mathcal N_i)\overset{\text{\ref{prop4}}}=N\otimes_R(\ilim{i} N_i)=\ilim{i}(N\otimes_R N_i)\overset{\text{\ref{prop4}}}=\ilim{i}\Hom_{\mathcal R}(\mathcal N^*,\mathcal N_i)$.}. 
Since $N_i$ is finitely generated, there exists $k>j$, such 
the composite morphism $\mathcal N_i\to \mathcal N_{i,sch}\to \mathcal N_j\to \mathcal N_k$ is equal to 
$f_{ik}$. Therefore, the composite morphism $\mathcal N_{i,sch}\to \mathcal N_j\to \mathcal N_k\to \mathcal N_{k,sch}$ is equal to 
$f_{ik}^{sch}$.
Then, 
$$\ilim{i}\mathcal N_{i,sch}=\ilim{i}\mathcal N_i =\mathcal M.$$

 $(2) \Rightarrow (3)$ It is obvious.
 
  $(3) \Rightarrow (1)$ Write $\mathcal M=\ilim{i} \mathbb N_i$, where every $\mathbb N_i$ is an $\R$-module scheme. Let $N\hookrightarrow N'$ be an injective morphism of $R$-modules. Then, the composite morphism

$$\aligned M\otimes_R N & =\bar{\mathcal M}(N)=\overline{\ilim{i}\mathbb N_i} (N)\overset{\text{\ref{2.16}}}=\ilim{i} \bar{\mathbb N_i}(N)\overset{\text{\ref{L5.112}}}\hookrightarrow \ilim{i} \bar{\mathbb N_i}(N')\overset{\text{\ref{2.16}}}=\overline{\ilim{i} \mathbb N_i}(N')\\ & =\bar{\mathcal M}(N')=M\otimes_R N'\endaligned$$
is injective and $M$ is a flat $R$-module.

\end{proof}

\begin{theorem} \label{3.5T} Let $R$ be a Noetherian ring.  Let $M$ be an $R$-module and $\{M_i\}$ the set of the finitely generated  submodules of $M$.
$M$ is a flat $R$-module iff $$\mathcal M=\ilim{i} \mathcal M_{i,sch}.$$ \end{theorem}

\begin{proof} $\Rightarrow)$  $M=\ilim{i} M_i$. Proceed as in the proof of $(1) \Rightarrow (2)$ in Proposition \ref{T3.7}.
%
%Every morphism $\mathcal M_i\to\mathcal M$ uniquely factors through $\mathcal M_i\to \mathcal M_{i,sch}$, by \ref{b}.
%Every morphism 
%$\mathcal V^*\to \mathcal M$ factors through $\mathcal M_j\to\mathcal M$, for some $j$, by \ref{Nota}. Then, $$\mathcal M=\ilim{i} \mathcal M_i=\ilim{i} \mathcal M_{i,sch}.$$

$\Leftarrow)$ It is a consequence of Theorem \ref{4.8z}.
\end{proof}

\begin{lemma} \label{3.3T} If $\mathcal  M=\ilim{i} \mathcal N_i^*$, then $$M\otimes_RN=\ilim{i} \Hom_R(N_i,N),$$ for any $R$-module $N$.
\end{lemma}

\begin{proof} Recall Definition \ref{notations}.Then,  $$M\otimes_RN=\bar{\mathcal M}(N)
=\overline{\ilim{i} \mathcal N_i^*}(N)=\ilim{i} (\overline{\mathcal N_i^*}(N))=
\ilim{i} \Hom_R(N_i,N)$$ for any $R$-module $N$. 

\end{proof}

\begin{corollary}  \label{NMP} Let $R$ be a Noetherian ring.  Let $M$ be an $R$-module and $\{M_i\}$ the set of finitely generated  submodules of $M$. $M$  is a flat $R$-module iff
$$M\otimes_R N=\ilim{i} \Hom_{R}(M_i^*,N),$$
for every $R$-module $N$.\end{corollary}

\begin{proof} $\Rightarrow)$ It is a consequence of \ref{3.5T} and \ref{3.3T}.

$\Leftarrow)$ $M\otimes -= \ilim{i} \Hom_{R}(M_i^*,-)$ is a left exact functor.
\end{proof} 

\subsection{From the category of affine $R$-schemes to the category of $X$-schemes} 

The reader can omit this subsection, no other section of this paper depends upon it. For simplicity, we have worked within the framework of the affine $R$-schemes. Briefly, let us see that we can work within the framework of the $X$-schemes.

Let $(X,O_X)$ be a Noetherian scheme. Given  a quasi-coherent sheaf, $F$, on $X$ (in the standard sense, see \cite[II 5.]{hartshorne}) we shall denote by $\mathcal F$ the functor from the category of quasi-compact and quasi-separated schemes over $X$ to the category of 
Abelian groups defined by ${\mathcal F}(Y):=(f^*F)(Y)$, for any quasi-compact and quasi-separated scheme $Y$ over $X$, $f\colon Y\to X$. Given a coherent sheaf $G$, we shall say that ${{\mathcal G}\,}^*:=\mathbb Hom_{ {\mathcal O}_X}({\mathcal G},{\mathcal O_X})$ is an ${\mathcal O}_X$-module variety. ${\mathcal F}$ and ${{\mathcal G}\,}^*$ are sheaves in the Zariski topos.

\begin{theorem} \label{elteorema} Let $(X,O_X)$ be a Noetherian scheme and ${F}$ a quasi-coherent module on $X$. $ F$ is a flat $ O_X$-module iff ${\mathcal F}$ is a direct limit of ${\mathcal O}_X$-module varieties.  
\end{theorem}

\begin{proof} $\Rightarrow)$ $F$ is equal to the direct limit of its coherent subsheaves, $ F=\ilim{i}  F_i$, and given an open set $U\subseteq X$ and a coherent subsheaf $ G$ of $F_{|U}$, there exists a coherent subsheaf $H\subseteq F$    such that $H_{|U}=G$ (see \cite[Ex. II 5.15]{hartshorne}). Given a coherent sheaf $I$, let $J:=I^*$. We shall denote $\mathcal I_{sch}=\mathcal J^*$. If $U=\Spec R\subseteq X$ is an affine open set, $M:=I(U)$ and $Y=\Spec {R'}$ is an affine scheme over $U$, then $\mathcal I_{sch}(Y)=\Hom_R(M^*,{R'})=\mathcal M_{sch}({R'})$.

Consider the natural morphism ${\mathcal F_i}\to \mathcal F_{i,sch}$, then the morphism
$$\Hom_{{\mathcal O}_X}(\mathcal F_{i,sch},{\mathcal F})\to \Hom_{{\mathcal O}_X}({\mathcal F_i},{\mathcal F})$$
is a bijection, by Lemma \ref{b}. Finally, the morphism
$$\ilim{i} \mathcal F_{i,sch}\to {\mathcal F}$$
is an isomorphism by Theorem \ref{3.5T}.

$\Leftarrow)$ Assume ${\mathcal F}=\ilim{i} {{\mathcal G}_i\,}^*$. Let $U=\Spec R\subseteq X$ be an affine open subset. 
Let $M:=F(U)$ and $N_i:=G_i(U)$. Then, $\mathcal M=\ilim{i} \mathcal N_i^*$. By \ref{T3.7}, $M$ is a flat $R$-module. Then, $F$ is flat.
\end{proof}

\section{Characterization of flat Mittag-Leffler modules}

\begin{lemma} \label{tonto} \begin{enumerate} \item Let $f\colon \mathcal N^*\to\mathcal M^*$ be  a morphism 
of $\mathcal R$-modules and let $f^*\colon \mathcal M\to\mathcal N$ be the dual morphism. Then,
$\Ker f=(\Coker f^*)^*$. Hence, $\Ker f$ is the module scheme associated with $\Coker f^*_R$, and
$f$ is a monomorphism iff $f^*$ is an epimorphism.

\item $\Hom_{\mathcal R}(\mathcal M^*,-)$ is a right exact funtor on the category of quasi-coherent modules.

\item Let $f_1\colon \mathcal N_1^* \to \mathcal N^*$ and $f_2\colon \mathcal N_2^*\to \mathcal N^*$ be two morphisms of $\mathcal R$-modules. Then, $\mathcal N_1^*\times_{\mathcal N^*} \mathcal N_2^*$ is an $\mathcal R$-module scheme associated with  an $R$-module and
$$(\mathcal N_1^*\times_{\mathcal N^*} \mathcal N_2^*)^*=\mathcal N_1\oplus_{\mathcal N} \mathcal N_2.$$

\end{enumerate}

\end{lemma}

 \begin{proof} (1) It is obvious.

(2) It is an immediate consequence of that $M\otimes_R-$ is a right exact functor on the category of $R$-modules.

 (3)  $\mathcal N_1\oplus_{\mathcal N} \mathcal N_2$ is a quasi-coherent $\mathcal R$-module because it is equal to the cokernel of the morphism $\mathcal N\to \mathcal N_1\oplus\mathcal N_2$, $n\mapsto (f_1^*(n),-f^*_2(n))$. Then, $\mathcal N_1^*\times_{\mathcal N^*} \mathcal N_2^*=
 (\mathcal N_1\oplus_{\mathcal N} \mathcal N_2)^*$ is a module scheme  and
 $$\mathcal N_1\oplus_{\mathcal N} \mathcal N_2=(\mathcal N_1\oplus_{\mathcal N} \mathcal N_2)^{**}=
(\mathcal N_1^*\times_{\mathcal N^*} \mathcal N_2^*)^*.$$

 \end{proof}

 \begin{definition}  \cite[Chap. 2 Def. 3.]{Raynaud} An $R$-module $M$ is said to be a flat Mittag-Leffler module  if $M$ is the direct limit of free finite $R$-modules $\{L_i\}$ such that the inverse system $\{L_i^*\}$
 satisfies the usual Mittag-Leffler condition (that is, for each $i$ there exists $j\geq  i$ such that for $k\geq j$ we have $\Ima(L_k^*\to L_i^*)=\Ima(L_j^*\to L_i^*)$).
 
\end{definition}

\begin{theorem} \label{ML} Let $M$ be an $R$-module. The following statements are equivalent

\begin{enumerate}
\item $M$ is a flat Mittag-Leffler module.

\item $\mathcal M$ is equal to a direct limit of  $\R$-submodule schemes associated with  modules.

\item Any morphism $\mathcal N^*\to \mathcal M$ factors through an $\R$-submodule scheme of $\mathcal M$ associated with an $R$-module.

\item Any morphism $\mathcal R^n\to \mathcal M$ factors through an $\R$-submodule scheme of $\mathcal M$ associated with an $R$-module.

\item The kernel of any morphism $\mathcal N^*\to \mathcal M$ 
is a module scheme associated with an $R$-module.
 
\item The kernel of any morphism $\mathcal R^n\to \mathcal M$ 
is a module scheme associated with an $R$-module.

\end{enumerate}
\end{theorem}

\begin{proof} $ (1) \Rightarrow (2)$ Write $M=\ilim{i} L_i$, where
the inverse system $\{L_i^*\}$ satisfies the Mittag-Leffler condition. Obviously, $M$ is a flat $R$-module and $\mathcal M^*=\plim{i} \mathcal L_i^*$. 
Let $N_i:=\Ima(L_k^*\to L_i^*)$, for $k>>i$. 
Obviously, the morphisms $N_i\to N_j$ are epimorphisms. The morphism $L_k^*\to N_i$ is an epimorphism for every $k>>i$, then the morphism
$\mathcal N_i^*\to \mathcal L_k$ is a monomorphism for every $k>>i$. Taking direct limits we have the morphisms 
$$\mathcal M=\ilim{i}\mathcal L_i \to \ilim{i}\mathcal N_i^*\hookrightarrow \mathcal M$$
Then, $\mathcal M=\ilim{i}\mathcal N_i^*$. 

$(2) \Rightarrow (3)$ It is a consequence of \ref{L5.11}.

$(3)\Rightarrow (1)$ Let $\mathcal V_1^*,\mathcal V_2^*$ be two $\mathcal R$-submodule schemes of $\mathcal M$. The morphism $\mathcal V_1^*\oplus \mathcal V_2^*\to \mathcal M$, $(w_1,w_2)\mapsto w_1+w_2$ factors through an $\mathcal R$-submodule scheme $\mathcal V^*$ of $\mathcal M$. Then, 
$\mathcal V_1^*,\mathcal V_2^*\subseteq \mathcal V^*$. Any morphism $\mathcal R^n\to \mathcal M$ factors through a submodule scheme of $\mathcal M$ asociated with an $\mathcal R$-module. Then, $\mathcal M$ is the direct limit of its $\mathcal R$-submodule schemes associated with modules. Hence, $M$ is flat by \ref{T3.7}.  Write $M=\ilim{i} L_i$, where $L_i$ is a finite free module for any $i$. We have to prove that
the inverse system $\{L_i^*\}$ satisfies the Mittag-Leffler condition.
The morphism $\mathcal L_i\to \mathcal M$ factors through a monomorphism $\mathcal W^*\hookrightarrow \mathcal M$. By \ref{L5.11}, there exists $j>$ such that $\mathcal W^*\to \mathcal M$ factors through a morphism $\mathcal W^*\to\mathcal L_j$, which is a monomorphism because $\mathcal W^*\hookrightarrow \mathcal M$ is a monomorphism. Likewise, the morphism $\mathcal W^*\to\mathcal L_k$ is a monomorphism for any $k\geq j$. Dually we have the morphisms
$$L_k^*\to W\to L_i^*.$$
and $L_k^*\to W$ is an epimorphism by \ref{tonto}. Then, $\Ima[L_k^*\to L_i^*]=\Ima[W\to L_i^*]$, for any $k\geq j$.

$(3) \Rightarrow (5)$ Let $f\colon \mathcal N^*\to \mathcal M$ be a morphism of $\mathcal R$-modules. $\Ima f$ is isomorphic to an $\mathcal R$-submodule of an $\mathcal R$-module scheme $\mathcal V^*\subseteq \mathcal M$. Consider the obvious morphisms $\mathcal N^*\to \Ima f\subseteq \mathcal V^*\subseteq \mathcal M$.  Then,
$$\Ker f=\Ker[\mathcal N^*\to \mathcal V^*]$$
is a module scheme associated with an $R$-module, by \ref{tonto}.

$(5) \Rightarrow (3)$ Let $f\colon \mathcal N^*\to\mathcal M$ a morphism of $\mathcal R$-modules.
Let $\mathcal V^*=\Ker f$. Consider the inclusion morphism $i\colon \mathcal V^*\hookrightarrow \mathcal N^*$. The dual morphism $i^*\colon \mathcal N\to \mathcal V$ is an epimorphism, by \ref{tonto}. Let $W:=\Ker i^*_R$. We have the exact sequence of morphisms
$$0\to W\to N\to V\to 0.$$
Then, $0\to \mathcal V^*\to \mathcal N^*\to \mathcal W^*\to 0$ is an exact sequence in the category of
module schemes associated with  $R$-modules. 
The dual morphism $f^*\colon \mathcal N^*\to\mathcal M$ factors through a morphism $g\colon \mathcal W^*\to \mathcal M$, because $\Hom_{\mathcal R}(-,\mathcal M)$ is a right exact functor on the category of module schemes associated with  $R$-modules, by \ref{tonto} (2).
Let $\mathcal H^*:=\Ker g$.
Observe that $$\mathcal N^*\times_{\mathcal W^*}  \mathcal H^*=\mathcal N^*\times_{\mathcal W^*}  \mathcal \Ker g=\Ker f =\mathcal V^*$$
Dually, $\mathcal N\oplus_{\mathcal W} \mathcal H=\mathcal V$, by \ref{tonto}, and 
$0= (\mathcal N\oplus_{\mathcal W} \mathcal H)/\mathcal V= \mathcal W\oplus_{\mathcal W} \mathcal H=\mathcal H$. Then, $\mathcal H=0$ and 
$g$ is a monomorphism.

$(3) \Rightarrow (4)$ It is obvious.

$(4) \Rightarrow (3)$ It is a consequence of the fact that any morphism $\mathcal N^*\to\mathcal M$ factors through a
morphism $\mathcal R^n\to \mathcal M$, by \ref{a}.

$(5) \Rightarrow (6)$ It is obvious.

$(6)\Rightarrow (5)$. Any morphism $f\colon \mathcal N^*\to\mathcal M$ factors through a
morphism $g\colon \mathcal R^n\to \mathcal M$, by \ref{a}. $\Ker f=\Ker g\times_{\mathcal R^n} \mathcal N^*$, which is a module scheme by \ref{tonto}.

\end{proof}
%
%The equivalence between of $(1)$ and $ (6)$ in \ref{ML} has the following immediate consequence.
%
%\begin{corollary} Let $M$ be an $R$-module. The following statements are equivalent
%
%\begin{enumerate}
%\item $M$ is a flat Mittag-Leffler module. 
%
%\item For any $R$-module $P$ and any $\sum_{i=1}^n m_i\otimes p_i\in M\otimes_R P$, $\sum_{i=1}^n m_i\otimes p_i=0$ iff
%$\sum_{i=1}^n w(m_i)\cdot p_i=0$, for every $w\in M^*$. \end{enumerate}\end{corollary}
%
%\begin{proof} FALSE The kernel of a morphism $f=(m_1,\ldots,m_n)\in \Hom_{\mathcal R}( \mathcal R^n,\mathcal M)=M^n$ is a module scheme associated with an $R$-module  iff  $\Ker f=\mathcal K^*$, where $K=\Coker f^*_R$.
%
%The sequence of morphisms $$0\to \mathcal K^*\to \mathcal R^n\overset{f}\to \mathcal M$$
%is exact iff the sequence of morphisms
%$$\aligned 0\to \Hom_R(K,P)\to \Hom_R(R^n,P) & \to M\otimes_R P\\ (p_i) & \mapsto \sum_i m_i\otimes p_i\endaligned$$
%is exact for any $R$-module $P$. 
%
%$K=R^n/\langle w(m_i)\rangle_{w\in M^*}$, 
%$\Hom_R(K,P)=\{(p_i)\in P^n\colon \sum_iw(m_i)p_i=0$, for any $w\in M^*\}$
%and $\Ker[\Hom_R(R^n,P) \to M\otimes_R P]=\{(p_i)\in P^n\colon \sum_{i=1}^n m_i\otimes p_i=0\}$. 
%
%We complete the proof of this proposition by  Theorem \ref{ML} (6).
%
%\end{proof}
%

\begin{corollary} \label{6.3}  Let $M$ be an $R$-module. The following statements are equivalent

\begin{enumerate}

\item $M$ is a flat Mittag-Leffler module.

\item For any $R$-module $N$ and  any morphism $f\colon \mathcal N^*\to \mathcal M$ there exists the smallest $\mathcal R$-submodule scheme $\mathcal W^*\subseteq \mathcal M$ such that $f$ factors through $\mathcal W^*$.

\item For any morphism $f\colon \mathcal R^n\to \mathcal M$ there exists the smallest $\mathcal R$-submodule scheme $\mathcal W^*\subseteq \mathcal M$ such that $f$ factors through $\mathcal W^*$.

\item For any free $R$-module $R^n$ and any $x\in R^n\otimes M$ there exist a submodule $W\overset i\subseteq R^n$ and $x'\in W\otimes_R M$ satisfying: \begin{enumerate} \item For any non-zero morphism $\phi\colon W\to V$, $(\phi\otimes Id)(x')\neq 0$ (where the morphism $\phi\otimes Id\colon W\otimes M\to V\otimes M$ is defined by $(\phi\otimes Id)(w\otimes m):=\phi(w)\otimes m$).

\item $(i\otimes Id)(x')=x$.\end{enumerate}

\item  For any $R$-module $N$ and  any morphism $f\colon \mathcal M^*\to \mathcal N$ there exists the smallest pair $(\mathcal N',f')$, where $N'$ is a submodule of $N$ and $f'\colon \mathcal M^*\to \mathcal N'$ a morphism of $\mathcal R$-modules, such that $f$ is the composite morphism $\mathcal M^*\overset{f'}\to \mathcal N'\to\mathcal N$.

\item For any $R$-module $N$ and any $x\in M\otimes_RN$, there exists the smallest pair  $(N',x')$ , where $N'$ is a submodule of $N$ and  $x'\in M\otimes_R N'$, such that  $x'$ maps to $x$ through the morphism $M\otimes_R N' \to M\otimes_R N$.

\end{enumerate}

\end{corollary}

\begin{proof} $(1) \Rightarrow (2)$ Let $g\colon M_1\to M_2$ be a morphism of $R$-modules. There exists a biggest quotient module of $M_1$ such that $g$ factors through it: $M/\Ker g$. Dually, given a morphism
$h\colon \mathcal M_2^*\to \mathcal M_1^*$ of $\mathcal R$-modules, there exists the smallest $\R$-submodule scheme of $\mathcal M_1^*$ such that $h$ factors through it. Given two submodule schemes $\mathcal W_1^*,\mathcal W_2^*\subset \mathcal M$, there exists a submodule scheme $\mathcal W^*\subset \mathcal M$ that contains them: a  submodule scheme $\mathcal W^*\subset \mathcal M$ that contains $\Ima[\mathcal W_1^*\oplus\mathcal W_2^*\to \mathcal M]$.
By \ref{ML} (3), we easily finish the proof.

$(2) \Rightarrow (1)$ It is an immediate consequence of \ref{ML} (3). 

$(2)\Rightarrow (3)$ It is obvious.

$(3)\Rightarrow (1)$   It is an immediate consequence of \ref{ML} (4). 

$(1) \Rightarrow (5)$ Let $f\colon \mathcal M^*\to\mathcal N$ be a morphism of $\R$-modules. Consider the dual morphism $f^*\colon \mathcal N^*\to \mathcal M$. 
$\Ker f^*$ is a module scheme,  by \ref{ML} (5).  
Write $\Ker f^*=\mathcal V^*$ and let  $i\colon\mathcal V^*\hookrightarrow \mathcal N^*$ be the inclusion morphism. $\mathcal V^*$ is the biggest submodule scheme of $\mathcal N^*$ such that 
$f^*\circ i=0$. Dually, $\mathcal V$ is the biggest quasi-coherent quotient module of $\mathcal N$ such that $i^*\circ f=0$. Then, $N'=\Ker i^*_R$ is the smallest submodule of $N$ such that $f$ factors through $\mathcal N'$.

 $(5) \Rightarrow (1)$ Let $f\colon \mathcal N_1\to \mathcal N_2$ be a morphism of $R$-modules
 and assume that $f_R$ is injective. 
 Let $g\colon \mathcal M^*\to \mathcal N_1$ be a morphism of $\mathcal R$-modules.
 If $f\circ g=0$,  $0$ is the smallest module such that $f\circ g$ factors through it, 
then  $0$ is the smallest module such that $g$ factors through it, 
 that is, $g=0$. By \ref{tonto} (2), $\Hom_{\mathcal R}(\mathcal M^*,-)$ is an exact functor on the category of quasi-coherent modules.
 
 Let $N'\subseteq N$ be the smallest submodule such that
 $f$ factors through a morphism $f'\colon \mathcal M^*\to \mathcal N'$. By \ref{ML} (3), it is sufficient to prove that
 $f'^*\colon \mathcal N'^*\to \mathcal M$ is a monomorphism. Let $w\in \Ker f'^*(S)\subseteq \mathcal N'^*(S)=\Hom_{\mathcal R}(\mathcal N',\mathcal S)$. 
 Then, $w\circ f'=0$ and $f'$ factors through the quasi-coherent module associated with $\Ker w_R$.
Then, $\Ker w_R=N'$ and $w=0$.  
 
$(5) \iff (6)$ It is an immediate consequence of the equality $\Hom_{\R}(\mathcal M^*,\mathcal N)=M\otimes_RN$. \end{proof}

\begin{proposition} Let $M$ be a flat Mittag-Leffler module. 
Then, the canonical morphism $$N^*\otimes_R M\to \Hom_R(N,M)$$
is injective for every $R$-module $N$.
\end{proposition}

\begin{proof}  Let $\{\mathcal N_j^*\}$ be the set of the reflexive $\R$-submodule affine schemes of $\mathcal M$. $\mathcal M=\ilim{j}\mathcal N_j^*$. Then,

$$\aligned N^*\otimes_R M & \overset{\text{\ref{1211b}}}=\Hom_{\R}(\mathcal N_{sch},\mathcal M)=\Hom_{\R}(\mathcal N_{sch},\ilim{j}\mathcal N_j^*)\overset{\text{\ref{L5.11}}}=\ilim{j}\Hom_{\R}(\mathcal N_{sch},\mathcal N_j^*)\\ & \overset{\text{\ref{1211}}}=\ilim{j}\Hom_{\R}(\mathcal N,\mathcal N_j^*)
\hookrightarrow \Hom_{\R}(\mathcal N,\mathcal M)=
\Hom_R(N,M)\endaligned$$\end{proof}

Let us prove some well known results (\ref{CML}, \ref{CMLC} and \ref{CMLD})
on Mittag-Leffler modules (see \cite{Raynaud} and \cite{RaynaudG}).

\begin{corollary} \label{CML} Let $M$ be a flat $R$-module. $M$ is a Mittag-Leffler module iff the natural morphism $M\otimes_R\prod_{i\in I} Q_i\to \prod_{i\in I} (M\otimes_R Q_i)$ is injective, for every set of $R$-modules $\{Q_i\}_{i\in I}$.

\end{corollary}

\begin{proof} $\Rightarrow)$ Let $\{\mathbb N_j\}$ be the set of the reflexive submodule affine schemes of $\mathcal M$. Then, $\mathcal M=\ilim{i}\mathbb N_i$ and

$$\aligned M\otimes_R \prod_i Q_i & \overset{\text{\ref{1.30}}}=\overline{\mathcal M}(\prod_i Q_i)\overset{\text{\ref{2.16}}}=\ilim{j} (\overline{\mathbb N_j}(\prod_i Q_i))\overset{\text{\ref{L5.112}}}=
\ilim{j} \prod_i (\overline{\mathbb N_j}( Q_i))\\ & \underset{*}\hookrightarrow \prod_i
\ilim{j}  (\overline{\mathbb N_j}( Q_i))\overset{\text{\ref{2.16}}}=\prod_i \overline{\mathcal M}(Q_i)\overset{\text{\ref{1.30}}}=\prod_i (M\otimes_R Q_i)\endaligned$$
($*$ Observe that the morphism $\bar {\mathbb N}_j\to \bar{\mathbb N}_{j'}$ is a monomorphism, for any $j'>j$).

$\Leftarrow)$ Let $f\colon \mathcal M^*\to \mathcal N$ be a morphism of $\R$-modules. Let $\{\mathcal N_i\}$ be the set of the submodules $N_i$ of $N$, such that the morphism $f$ factors through $\mathcal N_i\to \mathcal N$. Let $N':=\cap_{i} N_i$ and $N'_i:=N/N_i$. We have to prove that $f$ factors through $\mathcal N'\to \mathcal N$, by the equivalence of (1) and (3) in Corollary \ref{6.3} .
The obvious sequence of morphisms of modules
$$0\to N'\to N\to \prod_i N'_i$$
is exact. Then, the sequence of morphisms of modules
$$0\to M\otimes_RN'\to M\otimes_RN\to M\otimes_R\prod_i N'_i$$
is exact. The natural morphism $M\otimes_R\prod_{i\in I} N'_i\to \prod_{i\in I} (M\otimes_R N'_i)$ is injective. Then, the sequence of morphisms of modules
$$0\to \Hom_{\R}(\mathcal M^*,\mathcal N')\to \Hom_{\R}(\mathcal M^*,\mathcal N)\to \prod_i \Hom_{\R}(\mathcal M^*,\mathcal N'_i)$$
is exact, and $f$ factors through $\mathcal N'\to\mathcal N$.
\end{proof}

\begin{theorem} An $R$-module $M$ is a flat Mittag-Leffler module iff $\mathcal M$ is a direct limit of  $\R$-submodule affine  schemes.

\end{theorem}

\begin{proof} $\Rightarrow)$ It is a consequence of Theorem \ref{ML}.

$\Leftarrow)$ Write $\mathcal M=\ilim{j} \mathbb N_j$, where $\mathbb N_j
\hookrightarrow\mathcal M$ is an 
$\R$-submodule affine scheme, for any $j$.  $M$ is a flat $R$-module, by Theorem \ref{T3.7}.
Besides,

$$\aligned M\otimes_R \prod_i Q_i & \overset{\text{\ref{1.30}}}=\overline{\mathcal M}(\prod_i Q_i)\overset{\text{\ref{2.16}}}=\ilim{j} (\overline{\mathbb N_j}(\prod_i Q_i))\overset{\text{\ref{L5.112}}}=
\ilim{j} \prod_i (\overline{\mathbb N_j}( Q_i))\\ & \hookrightarrow \prod_i
\ilim{j}  (\overline{\mathbb N_j}( Q_i))\overset{\text{\ref{2.16}}}=\prod_i \overline{\mathcal M}(Q_i)\overset{\text{\ref{1.30}}}=\prod_i (M\otimes_R Q_i)\endaligned$$
Then, $M$ is a flat Mittag-Leffler module by Corollary \ref{CML}.

\end{proof}

An immediate consequence of Corollary \ref{CML} is the following corollary.

\begin{corollary} \label{CMLC} Let $f\colon M_1\to M_2$ be a universally injective morphism of $R$-modules. If $M_2$ is a flat Mittag-Leffler module then $M_1$ is a flat Mittag-Leffler module. If $M_1$ and $\Coker f$
are flat Mittag-Leffler modules then $M$ is a flat Mittag-Leffler module.\end{corollary}

\begin{examples} Free modules are flat Mittag-Leffler modules, then projective modules are flat Mittag-Leffler modules, because any projective module is a direct summand of a free module. $\mathbb Q$ is not a  flat Mittag-Leffler $\mathbb Z$-module: Consider the injective morphism
$\mathbb Z\to \prod_{n>0} \mathbb Z/n\mathbb Z$, $ m\mapsto (\bar m)_{n>0}$. 
$\mathbb Q\otimes_{\mathbb Z}  \prod_{n>0} \mathbb Z/n\mathbb Z\neq 0$ because $\mathbb Q$ is a flat $\mathbb Z$-module
and $\mathbb Q\otimes_{\mathbb Z}  \mathbb Z=\mathbb Q\neq 0$. $\prod_{n>0} (\mathbb Z/n\mathbb Z\otimes_{\mathbb Z}\mathbb Q)=0$. Then the morphism 
$$(\prod_{n>0} \mathbb Z/n\mathbb Z)\otimes_{\mathbb Z}\mathbb Q\to \prod_{n>0} (\mathbb Z/n\mathbb Z\otimes_{\mathbb Z}\mathbb Q)$$
is not injective and $\mathbb Q$ is not a flat Mittag-Leffler module.

\end{examples}

\begin{theorem} \label{CMLD} Let $M$ be a countably generated  $R$-module. $M$ is a flat Mittag-Leffler module iff $M$ is a projective module.\end{theorem}

\begin{proof} $\Rightarrow)$ Write $M=\langle m_n\rangle_{n\in\mathbb N}$.  $\mathcal M^*$ is the union of its reflexive submodule affine schemes, because $M$ is a flat Mittag-Leffler module.
Define $\mathcal N_n^*$ recursively as a submodule scheme of $\mathcal M$ such that
$m_n\in \mathcal N^*(R)$ and $\mathcal N_{n-1}^*\subseteq \mathcal N_n^*$.  Then,
$\mathcal M=\ilim{n\in\mathbb N} \mathcal N_n^*$. 
The obvious morphisms $\pi_n\colon \mathcal N_n\to \mathcal N_{n-1}$ are epimorphisms, by Lemma \ref{tonto}. 

%$\Coker \pi_n$ is a quasi-coherent module and $(\Coker \pi_n)^*$ is the kernel of the monomorphism $\mathcal N_{n-1}^*\hookrightarrow \mathcal N_n^*$, then $(\Coker \pi_n)^*=0$ and $(\Coker \pi_n)^*=0$.

Let $f\colon P\to P'$ be an epimorphism of $R$-modules. The morphisms $N_n\otimes_R P\to N_n\otimes_R P'$ are epimorphisms and, then,  it is easy to prove that the morphism $\plim{n\in\mathbb N} (N_n\otimes_R P)\to \plim{n\in\mathbb N} (N_n\otimes_R P')$ is an epimorphism. Then, from the commutative diagram
$$\xymatrix@R8pt{ \Hom_R(M,P) \ar@{=}[d]^-{\text{\ref{tercer}}} \ar[r] & \Hom_R(M,P') \ar@{=}[d]^-{\text{\ref{tercer}}} \\
\Hom_{\R}(\mathcal M, \mathcal P) \ar@{=}[d] \ar[r] & \Hom_{\R}(\mathcal M,\mathcal P') \ar@{=}[d] \\ \Hom_{\R}(\ilim{n\in\mathbb N}\mathcal N_n^*,\mathcal P) \ar@{=}[d] \ar[r] & \Hom_{\R}(\ilim{n\in\mathbb N} \mathcal N_n^*,\mathcal P') \ar@{=}[d] \\
\plim{n\in\mathbb N} \Hom_{\R}(\mathcal N_n^*,\mathcal P) \ar@{=}[d]^-{\text{\ref{prop4}}} \ar[r] & \plim{n\in\mathbb N}\Hom_{\R}(\mathcal N_n^*,\mathcal P') \ar@{=}[d]^-{\text{\ref{prop4}}} \\ \plim{n\in\mathbb N} ( N_n\otimes_R  P)  \ar[r] & \plim{n\in\mathbb N} (N_n\otimes_R P') }$$
we have that $\Hom_R(M,P)\to\Hom_R(M,P')$ is an epimorphism, hence $M$ is a projective module.

$\Leftarrow)$ Any projective module is a flat Mittag-Leffler module.

\end{proof}

Any projective module is a direct sum of countably generated projective modules (see \cite{Kaplansky}), then any projective module is a direct sum of countably generated  flat Mittag-Leffler modules.

\section{Appendix: Sheaves}  

 Let $\{U_i\}_{i\in I}$ be an open covering of a scheme $X$. We shall say that the obvious morphism $Y=\coprod_{i\in I} U_i\to X$ is an open covering. 

Lemma \ref{2.11} is well known, but it is difficult to find a reference with exactly the same hypothesis.

\begin{definition} Let $\mathbb F$ be a functor of sets.  
$\mathbb F$  is said to be a sheaf in the Zariski topos if 
for any commutative $R$-algebra ${R'}$ and any  open covering $\Spec {R'}_1\to \Spec {R'}$, the sequence of morphisms
$$\xymatrix{ \mathbb F({R'})\ar[r] &  \mathbb F({R'}_1)  \ar@<1ex>[r]  \ar@<-1ex>[r] & 
\mathbb F({R'}_1\otimes_{R'} {R'}_1)}$$
is exact,  that is to say by the Yoneda Lemma,  the sequence of morphisms
$$\xymatrix @C=12pt { \Hom((\Spec {R'})^\cdot,\mathbb F) \ar[r] &  \Hom((\Spec {R'}_1)^\cdot,\mathbb F)   \ar@<1ex>[r]  \ar@<-1ex>[r] & 
\Hom((\Spec {R'}_1\underset{\Spec {R'}}\times\Spec {R'}_1)^\cdot,\mathbb F) }$$
is exact.

\end{definition}

\begin{examples} If $M$ is an $R$-module, $\mathcal M$ is a sheaf in the Zariski topos. 
If $X$ is an $R$-scheme, $X^\cdot$ is a sheaf in the Zariski topos. If  $\{\mathbb F_i\}$ a direct system of  sheaves in the Zariski topos, $\ilim{i} \mathbb F_i$ is a sheaf in the Zariski topos. \end{examples}

Given a morphism $h\colon X\to Y$ of $R$-schemes, let $h^\cdot\colon X^\cdot \to Y^\cdot$ be the induced morphism.

\begin{lemma} \label{2.11} Let $X$ be a scheme and $\pi\colon X_1\to X$ an open covering. Let $\mathbb F$ be a sheaf in the Zariski topos. Then, the sequence of morphisms
$$\xymatrix{ \Hom(X^\cdot,\mathbb F) \ar[r] &  \Hom(X_1^\cdot,\mathbb F)   \ar@<1ex>[r]  \ar@<-1ex>[r] & 
\Hom(X_1^\cdot\times_{X^\cdot} X_1^\cdot,\mathbb F) }$$
is exact.

\end{lemma}

\begin{proof} Recall
$\Hom((\Spec {R'})^\cdot,\mathbb F)=\mathbb F({R'})$ by the Yoneda Lemma.

Let $\pi_1,\pi_2 \colon  X_1\times_X X_1 \dosflechas X_1$ be the obvious projections.  

Let $f,f' \in \Hom(X^\cdot,\mathbb F) $ and 
  $x\in X^\cdot({R'})=\Hom_{\Spec R}(\Spec {R'},X)$. Suppose 
$f_{R'}(x)\neq f'_{R'}(x)$, that is to say, $f\circ x^\cdot \neq f'\circ x^\cdot$. There exist an open covering  $y\colon \Spec {R'}_1\to \Spec {R'}$  and $x_1\in 
X_1^\cdot({R'}_1)= \Hom_{\Spec {R'}}(\Spec {R'}_1,X_1)$ such that the diagram
$$\xymatrix{\Spec {R'}_1 \ar[d]_-{x _1} \ar[r]^-y & \Spec {R'} \ar[d]_-{x}  \\ X_1 \ar[r]_-\pi & X}$$
is commutative.  Observe that $f\circ x^\cdot \circ y^\cdot \neq f'\circ x^\cdot\circ y^\cdot$ because
the morphism $$\Hom((\Spec {R'})^\cdot,\mathbb F)=\mathbb F({R'})\to \mathbb F({R'}_1)=\Hom((\Spec {R'}_1)^\cdot,\mathbb F)$$ is injective. Then, $f\circ \pi^\cdot \circ x_1^\cdot \neq f'\circ \pi^\cdot \circ x_1^\cdot$
and $f\circ \pi^\cdot \neq f'\circ \pi^\cdot $. Therefore,  $$\Hom(X^\cdot,\mathbb F) \to \Hom(X_1^\cdot,\mathbb F)$$ is injective.

Let $f_1\in  \Hom(X_1^\cdot,\mathbb F) $ be a morphism such that $f_1\circ \pi_1^\cdot=f_1\circ \pi_2^\cdot$.
Given a morphism $x\in X^\cdot({R'})=\Hom(\Spec {R'},X)$, there exist an open covering $y\colon \Spec {R'}_1\to \Spec {R'}$ and $x_1\in 
X_1^\cdot({R'}_1)=\Hom(\Spec {R'_1},X)$ such that the diagram
$$\xymatrix{\Spec {R'}_1 \ar[d]_-{x _1} \ar[r]^-y & \Spec {R'} \ar[d]^-{x}  \\ X_1 \ar[r]_\pi & X}$$
is commutative.  There exists a unique morphism $g\colon (\Spec {R'})^\cdot\to\mathbb F$
such that the diagram
$$\xymatrix{(\Spec {R'}_1\times_{\Spec {R'}}\Spec {R'}_1)^\cdot \ar@<1ex>[r]  \ar@<-1ex>[r]  \ar[d]_-{x_1^\cdot \times x_1^\cdot} & (\Spec {R'}_1)^\cdot \ar[d]_-{x_1^\cdot}  \ar[r] & (\Spec {R'})^\cdot \ar[d]_-{x^\cdot}  \ar@/^5pc/[ldd]_-{\,\,g}
\\
X_1^\cdot\times_{X^\cdot} X_1^\cdot \ar@<1ex>[r]^-{\pi_1^\cdot}  \ar@<-1ex>[r]_-{\pi_2^\cdot} \ar[rd] & X_1^\cdot  \ar[d]^{f_1} \ar[r]^-{\pi^\cdot} & X^\cdot\\  & \mathbb F & }$$
is commutative, because $\mathbb F$ is a sheaf. Let $f\colon X^\cdot \to \mathbb F$ be defined by $f_{R'}(x):=g$.
This morphism is well defined, that is, $g$ does not depend on the choice of $y$ and $x_1$: Let $y_2\colon \Spec R'_2\to \Spec R$ be an open covering  and a commmutative diagram
$$\xymatrix{\Spec {R'}_2 \ar[d]_-{x _2} \ar[r]^-{y_2} & \Spec {R'} \ar[d]^-{x}  \\ X_1 \ar[r]_\pi & X}$$
Consider the obvious open covering $y_3\colon \Spec (R'_1)\coprod  \Spec(R'_2)\to \Spec R$ and the obvious morphism $x_3\colon \Spec (R'_1)\coprod \Spec (R'_2)\to  X_1$. It is easy to prove that the morphism $(\Spec R')^\cdot\to \mathbb F$ defined by $y$ and $x_1$ is equal to the morphism defined by  $y_3$ and $x_3$, and that the  morphism defined
by $y_3$ and $x_3$ is equal to the  morphism defined by  $y_2$ and $x_2$.

Finally, $f_1=f\circ \pi^\cdot$: Let $x'\in X_1^\cdot(R')=\Hom(\Spec R',X_1)$ and let $x:=\pi\circ x'\in \Hom(\Spec R',X_1)$. Let us follow previous notations. Then, $R'_1=R'$,  $x_1=x'$ and $g=f_1\circ {x'}^{\cdot}$. Therefore,
$(f\circ \pi^\cdot)_{R'} (x')=f_{R'}(x)=g=f_1\circ {x'}^{\cdot}={f_1}_{R'}(x')$.

\end{proof} 

Observe that the morphism $\Hom(X^\cdot,\mathbb F) \to \Hom(X_1^\cdot,\mathbb F)$  is injective, then 
for any open covering $X_2\to X_1\times_X X_1$, the sequence of morphisms
$$\xymatrix{ \Hom(X^\cdot,\mathbb F) \ar[r] &  \Hom(X_1^\cdot,\mathbb F)   \ar@<1ex>[r]  \ar@<-1ex>[r] & 
\Hom( X_2^\cdot,\mathbb F) }$$
is exact.

\begin{proposition} \label{2.10} \label{L5.11} Let $X$ be an $R$-module  quasi-compact and quasi-separated  scheme and $\{\mathbb F_i\}$ a direct system of  $\R$-modules. Assume $\mathbb F_i$ is a sheaf in the Zariski topos, for any $i$. Then,
$$\Hom_{\R}(X^\cdot, \ilim{i} \mathbb F_i )=\ilim{i} \Hom_{\R}(X^\cdot, \mathbb F_i ).$$
\end{proposition}

\begin{proof} If $X=\Spec A$ is an affine scheme,  $$
\Hom(X^\cdot,\ilim{i} \mathbb F_i )=(\ilim{i} \mathbb F_i )(A)=\ilim{i} \mathbb F_i (A)= \ilim{i} \Hom(X^\cdot,\mathbb F_i )$$ by the Yoneda Lemma.
Let $X_1=\Spec A_1\to X$ be an affine open covering and let $X_2=\Spec A_2\to X_1\times_X X_1$ be an affine open covering.  Consider the commutative diagram 
$$\xymatrix{\Hom(X^\cdot,\ilim{i} \mathbb F_i ) \ar[r] & \Hom(X_1^\cdot,\ilim{i} \mathbb F_i )  \ar@<1ex>[r]  \ar@<-1ex>[r] \ar@{=}[d] & \Hom(X_2^\cdot,\ilim{i} \mathbb F_i ) \ar@{=}[d] \\ \ilim{i} \Hom(X^\cdot, \mathbb F_i ) \ar[r] & \ilim{i} \Hom(X_1^\cdot, \mathbb F_i ) \ar@<1ex>[r]  \ar@<-1ex>[r]  & \ilim{i} \Hom(X_2^\cdot, \mathbb F_i )}$$
Then, $\Hom(X^\cdot,\ilim{i} \mathbb F_i )=\ilim{i} \Hom(X^\cdot,\mathbb F_i )$.

Let us prove that  $\Hom_{\R}(X^\cdot,\ilim{i} \mathbb F_i )=
\ilim{i} \Hom_{\R}(X^\cdot,\mathbb F_i )$: Let $f\in \Hom_{\R}(X^\cdot,\ilim{i} \mathbb F_i )$. Consider the commutative diagram
$$\xymatrix{ X^\cdot \times X^\cdot \ar[r]^-+ \ar[d]^-{f\times f} & X^\cdot \ar[d]^-f\\
\ilim{i} \mathbb F_i \times \ilim{i} \mathbb F_i  \ar[r]^-+ & \ilim{i} \mathbb F_i }$$
Then, there exists a morphism $f_k\colon X^\cdot \to \mathbb F_k$ 
such that $f$ is equal to the composite morphism $X^\cdot \overset{f_k} \to \mathbb F_k\to \ilim{i}\mathbb F_i $ and the diagram
$$\xymatrix{ X^\cdot \times X^\cdot \ar[r]^-+ \ar[d]^-{f_k\times f_k} & X^\cdot \ar[d]^-{f_k}\\
\mathbb F_k\times  \mathbb F_k \ar[r]^-+ &  \mathbb F_k}$$
is commutative, i.e., $f_k$ is linear. Likewise, we can prove that $f_k$ is $\R$-linear.

\end{proof}

\end{document}